\documentclass[preprint,12pt]{elsarticle}
\usepackage{amsmath}
\usepackage{amsthm}
\newtheorem{definition}{Definition}[subsection]
\newtheorem{lemma}{Lemma}[subsection]
\newtheorem{example}{Example}[subsection]
\newtheorem{corollary}{Corollary}[subsection]
\newtheorem{remark}{Remark}[subsection]
	\usepackage{orcidlink}
\usepackage{amssymb}
\usepackage{hyperref}
\usepackage{float}
\usepackage{tabularx} 
\journal{Chaos, Solitons \& Fractals}

\begin{document}

	\begin{frontmatter}
		
		
		
		\title{A Descriptive Perspective on Devaney's Chaos and Some Results on Topologically Conjugate Systems}
		

	\author{Fatih Uçan \corref{cor1} \orcidlink{0000-0002-6975-9408}} 
	\ead{fatihucan@ktu.edu.tr}
	
	\author{Tane Vergili \orcidlink{0000-0003-1821-6697}}
		\ead{tane.vergili@ktu.edu.tr}
		
	\affiliation{organization={Karadeniz Technical University, Department of Mathematics},
		city={Trabzon},
		postcode={61000},
		country={Türkiye}}

		%% Abstract
		\begin{abstract}
	In this study, Devaney’s chaos conditions are revisited within the framework of descriptive proximity. The concepts of descriptive transitivity, the density of descriptive periodic objects, and descriptive sensitivity are defined. The most notable finding of the study is that Banks Theorem, which establishes the hierarchy among these conditions in classical topology, does not generally hold in the descriptive perspective, and some of the concepts above remain invariant under topological conjugacy certain conditions.
		\end{abstract}

		%% Keywords
		
		\end{frontmatter}
	
		
		
		\section{Introduction}
		\label{sec1}
		%% Labels are used to cross-reference an item using \ref command.

		Given metric space $X$, a map $f: X \rightarrow X$ is said to be chaotic in the sense of Devaney  for the discrete dynamical system $(X,f)$ if it has sensitive dependent on initial conditions, a dense set of periodic points, and is topologically transitive \cite{devaney2018introduction}. The conditions of Devaney's chaos are well-suited for the topological perspective and  continue to be investigated \cite{LIANG2025116826,MEZZI2025109571}. Furthermore, these conditions are interdependent; as demonstrated in \cite{banks1992devaney}, the density of periodic points and topological transitivity together imply sensitive dependence on initial conditions, rendering the first condition redundant in many cases. A classical example is Arnold's Cat transformation, which exhibits strong mixing properties, dense periodic points, and topological transitivity on the torus \cite{Arnold1968}. Due to these properties, it is widely used as a benchmark model in the study of chaotic dynamical systems.
		In the study of dynamical systems, certain properties, such as Devaney’s criteria for chaos, remain invariant under topological conjugacy. Two systems, denoted by $(X, f)$ and $(Y, g)$, are considered to be topologically semi-conjugate if there exists a continuous and surjective map $h: X \to Y$ that satisfies the commutativity condition
		$h \circ f = g \circ h$. If $h$ is also a homeomorphism, then $f$ and $g$ are said to be topologically conjugate \cite{morris1989topology}. Topological conjugacy provides a fundamental bridge between two dynamical systems. It implies that a complex system shares the same dynamical characteristics as its conjugate, which is often a simpler or more well-understood system. Consequently, conjugacy allows for the analysis of intricate dynamics by transforming them into a more tractable form without loss of essential information. \newline

		In a metric space, the proximity of two subsets is defined by their distance being zero. In the more general setting of topological spaces, two sets are considered close if their closures have a non-empty intersection. James F. Peters takes this concept of proximity to a new level \cite{peters2007near} and addressed the question: 'Can two subsets be considered close even if they are not spatially near, provided that they share common descriptive features?' The concept of "descriptive proximity" has been introduced into the literature to describe the situation where entities are spatially remote while descriptively near.\\
		
		Assume that the elements of a dynamical system consist of non-abstract points; in other words, they possess descriptive features that are vector-valued. In this framework, these elements are referred to as 'objects.' Under the iteration of the dynamical system, these objects modify their descriptive features accordingly. Consequently, the evolution of these descriptive characteristics can be analyzed to determine whether the underlying motion exhibits chaotic behavior. This work investigates the conditions under which a function exhibits chaos in a descriptive sense. To this end, descriptive transitivity, the density of descriptive periodic objects, and  descriptive sensitivity are examined, and illustrative examples are provided. Banks' Theorem states that in the classical definition of chaos, topological transitivity and the density of periodic points together imply sensitive dependence on initial conditions. However, this is not the case for the construction of descriptively chaotic functions: in a descriptive sense, transitivity and the density of periodic objects do not necessarily imply descriptive sensitivity. 
		The examples in the table following are given for this purpose.

\begin{table}[h]
	\centering
	\renewcommand{\arraystretch}{1.5}

	\begin{tabularx}{\textwidth}{| >{\raggedright\arraybackslash}p{4.5cm} | >{\centering\arraybackslash}X | >{\centering\arraybackslash}X | >{\centering\arraybackslash}X |} \hline
		
		& \textbf{Example \ref{exTQ}} & \textbf{Example \ref{exTI}} & \textbf{Example \ref{doublingsensitve}} \\ \hline
		
		Topological transitive & $\mathsf{X}$ & \checkmark & \checkmark \\ \hline
		
		Descriptively transitive & \checkmark & \checkmark & \checkmark \\ \hline
		
		 Dense periodic objects (points) set & \checkmark & $\mathsf{X}$ &\checkmark \\ \hline
		
		Dense descriptive periodic objects set & \checkmark & \checkmark & \checkmark \\ \hline

	\end{tabularx}
\end{table}

		\section{Preliminaries}
		Let us assume that the space $X$ (or, roughly speaking, the set $X$) consists of non-abstract elements, that is, its elements possess measurable properties. Each measurable property here corresponds to a function $\phi: X \to \mathbb{R}$.  For example, if $X$ is a digital image, then for a pixel $x \in X$, $\phi(x)$ could represent the color, brightness, or intensity of that pixel. Here, the function $\phi$ is called a probe function on $X$, and the value $\phi(x) \in \mathbb{R}$ is referred to as the feature value of the point $x$. Elements of the set may have more than one descriptive feature \cite{peters2013near}. In this study, we will assume that these features are finite (i.e., there are a finite number of probe functions $\phi_1, \phi_2, \ldots, \phi_n$ defined on $X$).
		Thus, for an element $x \in X$, the vector $\Phi(x) = (\phi_1(x), \ldots, \phi_n(x)) \in \mathbb{R}^n$ is the feature vector of the point $x$, and for $A \subset X$, 
		$\Phi(A) =\{\Phi(a)   \, \, |\, a\in A\}\subseteq \mathbb{R}^n$ is the set of feature vectors of the subset $A$.  The descriptive similarity of subsets A and B is defined by the non-empty intersection of the feature value sets $\Phi(A)$ and $\Phi(B)$ \cite{naimpally2013foreword}. Two subsets  $A$ and $ B$  are descriptively near, denoted by $A \delta_{\Phi} B$, if their corresponding sets of feature vectors have a non-empty intersection, i.e.,
		
		\[  A \delta_{\Phi} B \ \ \ :\Leftrightarrow \ \ \ \Phi(A) \cap \Phi(B) \neq \emptyset.
\]	
	\noindent  
If $\Phi(A)$ and $\Phi(B)$ do not overlap, we say that $A$ and $B$  are descriptively far which is denoted by $A\not{\delta}_{\Phi}B$ \cite{book}. This method allows us to examine whether sets are descriptively similar  or not, despite being spatially far, owing to their shared descriptive features \cite{naimpally2013topology}. Additionally, it's clear that the descriptive proximity relation satisfies reflexivity and symmetry properties.
\\

\noindent Let $X$ be a set, such as a genus-2 surface, equipped with a specific descriptive property.  Consider $W = \{ A, A^{\prime}, B, B^{\prime}, C, C^{\prime} \} \subset 2^X $ with a probe function $\Phi=\phi : W \rightarrow \mathbb{R} $ such that the following color wavelengths are assigned respectively
\begin{align*}
	\Phi(A) &= 617 \text{ nm}, \\
	\Phi(A^{\prime}) &= 510 \text{ nm}, \\
	\Phi(B) &= 639 \text{ nm}, \\
	\Phi(B^{\prime}) &= 411 \text{ nm}, \\
	\Phi(C) &= 480 \text{ nm}, \\
	\Phi(C^{\prime}) &= 617 \text{ nm}.
\end{align*}

\begin{figure}[H]
	\centering
	\includegraphics[width=1\linewidth]{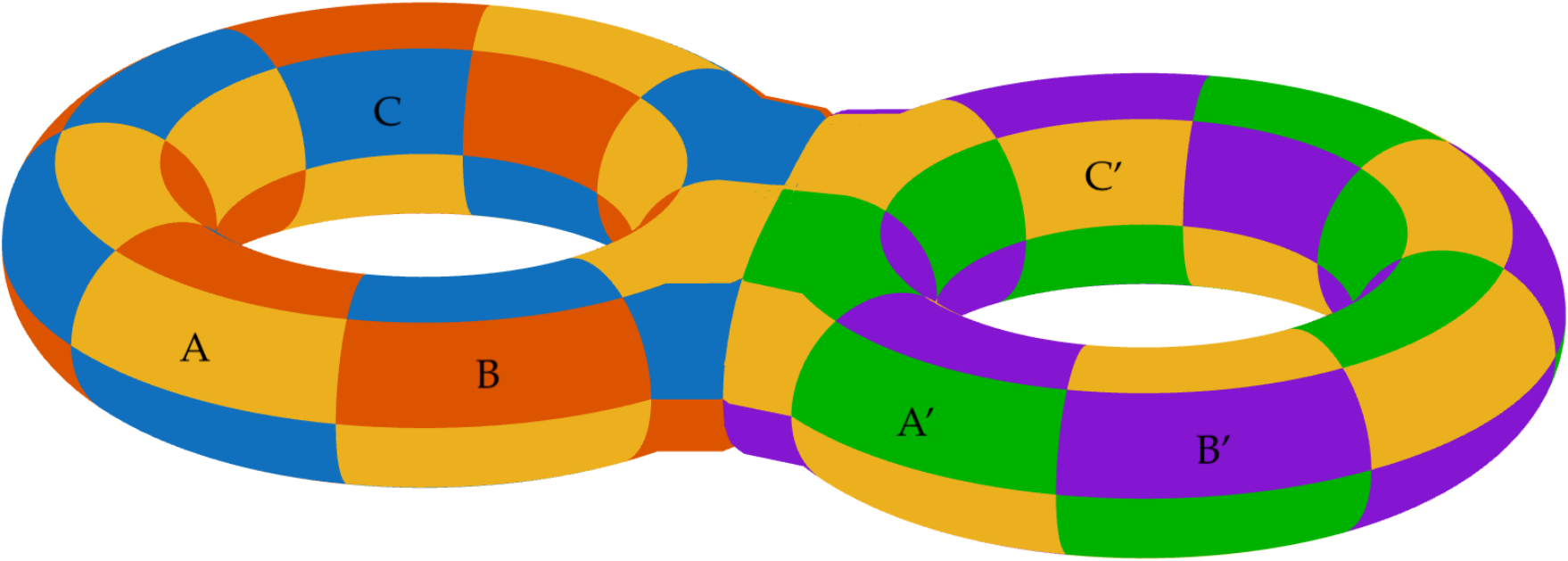}
	\caption{Genus-2 like a surface}
	\label{fig:toruss}
\end{figure}
\noindent The comparison of objects results is presented in the table below.

%% Use \subsubsection, \paragraph, \subparagraph commands to 
%% start 3rd, 4th and 5th level sections.
%% Refer following link for more details.
% TODO: \usepackage{graphicx} required

%% https://en.wikibooks.org/wiki/LaTeX/Document_Structure#Sectioning_commands

\begin{table}[H]
	\centering
	\renewcommand{\arraystretch}{1.2} % Satırları biraz ferahlatır
	\begin{tabular}{c | c c}
		
		\textbf{Descriptive Near Sets} &  \textbf{Descriptive Far Sets} \\ \hline
		$A \delta_{\Phi} A$             & $A \not\delta_{\Phi} A'$ 	& $B \not\delta_{\Phi} B'$ \\
		$A \delta_{\Phi} C'$            &  $A' \not\delta_{\Phi} B'$ 	& $B \not\delta_{\Phi} C$ \\
		$A' \delta_{\Phi} A'$           & $A' \not\delta_{\Phi} C'$  & $B \not\delta_{\Phi} C'$ \\
		$B \delta_{\Phi} B$             & $A \not\delta_{\Phi} B'$ 	& $B' \not \delta_{\Phi} C' $  \\
		$C \delta_{\Phi} C$             & $A \not\delta_{\Phi} C'$ &$B \not\delta_{\Phi} A'$  \\
		
		$C' \delta_{\Phi} C'$           & $A \not\delta_{\Phi} B$ &$C \not\delta_{\Phi} B'$\\
		$B' \delta_{\Phi} B'$           & $A \not\delta_{\Phi} C$ 	& $C \not\delta_{\Phi} A'$ \\
		&  $C \not\delta_{\Phi} C'$ &

	\end{tabular}
	\caption{Descriptive comparison objects of W.}
	\label{tab:near_far_sets}
\end{table}
  \noindent Di Concilio et al. \cite{Di_Concilio2018-oy} introduce the descriptive intersection set $A\underset{\Phi}{\cap} B$, which is formally defined as:$$ A\underset{\Phi}{\cap} B=\{ x\in A\cup B \, | \, \Phi(x)\in \Phi(A) \cap \Phi(B)\} $$for any $A, B \in 2^X$. This mathematical construction serves as a robust tool for identifying common characteristics between objects in sets $A$ and $B$, even when they do not overlap spatially. \\

		\noindent  Given two set of objects $X$ and $Y$ together with probe functions $\Phi_1 : X \rightarrow \mathbb{R}^n$ and $\Phi_2 : Y \rightarrow \mathbb{R}^n$, respectively, a function $h: X\rightarrow $ Y is said to be descriptively continuous, provided  $A\delta_{\Phi_1} B$ implies $ h(A)\delta_{\Phi_2} h(B)$ for all $A,B\in 2^X$ \cite{HAIDER2021111237}. \\
		
		\noindent In the case where $X$ is a topological space, it is a common misconception that continuity implies descriptive continuity, but this is not the case. For instance in standart metric space given as the function $f:\mathbb{R} \rightarrow \mathbb{R},  f(x)=x+1$ and the probe function $\Phi=\phi: \mathbb{R} \rightarrow \mathbb{R}$ 
		\[\phi(x) =
		\begin{cases} 
			r_1, &  x\in [0,1] \\
			r_2,&   otherwise
		\end{cases}
		\]
		where $r_1,r_2 \in \mathbb{R}$ such that $r_1\neq r_2$.
		It is clear that \( f \) is continuous because it is an isometry, however $ f $ is not descriptively continuous since,
		$\{1\}\not{\delta_{\Phi}} \{\frac{3}{2}\}$ while $\{0\}  \delta_{\Phi} \{\frac{1}{2}\}$ for $A=\{0\}$, $B=\{\frac{1}{2}\} \in 2^\mathbb{R}.$

		%\subsubsection{Mathematics}
		%% Inline mathematics is tagged between $ symbols.
		%This is an example for the symbol $\alpha$ tagged as inline mathematics.
		
		%% Displayed equations can be tagged using various environments. 
		%% Single line equations can be tagged using the equation environment.
		%\begin{equation}
		%f(x) = (x+a)(x+b)
		%\end{equation}
		
		%% Unnumbered equations are tagged using starred versions of the environment.
		%% amsmath package needs to be loaded for the starred version of equation environment.
		%\begin{equation*}
		%(x) = (x+a)(x+b)
		%\end{equation*}
		
		%% align or eqnarray environments can be used for multi line equations.
		%% & is used to mark alignment points in equations.
		%% \\ is used to end a row in a multiline equation.
		%\begin{align}
		% f(x) &= (x+a)(x+b) \\
		%      &= x^2 + (a+b)x + ab
		%\end{align}
		 \section{A descriptive perspective on Devaney's chaos}

Throughout this paper, a discrete dynamical system $(X, f)$ together with a probe function $\Phi: X \rightarrow \mathbb{R}^n$ will be represented as a triple  $(X, f, \Phi)$ called  a descriptively dynamical system. Specifically if  a set $A=\{ a \}$ is singleton, one can prefer the term 'object' instead of 'set' so that $\Phi(\{a\})=\Phi(a)$.\linebreak In that case, 
		descriptive concepts can be given as follows:
		\begin{center}
			
			\begin{itemize}
				\item[(i)] The set denoted as $O_f(a,\Phi)$ represents the descriptive orbit of an object $a \in X$ under the function $f:X\rightarrow X$,	
				\begin{align*}
			   O_f(a,\Phi) &=\{\Phi(f^j(a)) \; | \; a \in X , j\in \mathbb{Z}^+\} \\ & =\{\Phi(a),\Phi(f(a)),\dots,\Phi(f^j(a)),\dots\}.
			\end{align*}
				
				\item[(ii)] A object $a \in X$ is a descriptive fixed object of the function $f$, provided that $\Phi(f(a)) = \Phi(a) $.
				
				\item[(iii)] A object $a \in X$ is a descriptive $m$-periodic object of $f$ for $m\in \mathbb{Z}^+$,  provided that  \[\Phi(f^{j}(a)) \neq \Phi(a) \text{ for } j=1,2,\dots,m-1 \text{ and }  \Phi(f^{m}(a)) = \Phi(a). \]
			\end{itemize}
		\end{center}
		
		\noindent The sets of periodic objects can generally be given as follows:  
		
		\begin{center}
			\begin{itemize}
				\item [(iv)]  $Per_m(f,\Phi)=\{a\in X \; | \; \Phi(f^m(a))=\Phi(a)\}$ is the set of descriptive $m$-periodic objects.
				\item  [(v)]
				$Per(f,\Phi)=\underset{m\in\mathbb{Z}^+}{\bigcup}Per_m(f,\Phi)=\{a\in X \; | \; \Phi(f^m(a))=\Phi(a), \; m\in \mathbb{Z}^+\}$ is the set of all descriptive periodic objects.
			\end{itemize}
		\end{center}
	 The question, 'If a chaotic system possesses descriptive properties, does the system's dynamism affect its own descriptive properties?', serves as the core motivation for our study.	Let $X$ be a topological space, and let the triple $(X, f, \Phi)$ a descriptive dynamical system. From this perspective, we will examine Devaney’s chaos conditions within this space and discuss the obtained results.

		\begin{lemma}\label{lem21} The density of the set of periodic objects of  $f$ implies the density of the set of descriptive periodic objects.
			\begin{proof}
				Let $U$ be a nonempty open subset of $X$. Since the set of periodic objects of the function $f$ is dense in $X$, there exists an $a \in U$ such that $f^k(a) = a$ for some integer $k > 0$. In this case, by the equality $\Phi(f^k(a)) = \Phi(a)$,  $a$ is also a descriptive periodic object of $f$, and hence the desired result is achieved.
			\end{proof}
		\end{lemma}

		\begin{definition}
			 $f$ is said to be descriptively transitive if for every nonempty open sets $U, V \subseteq X$, there exists an integer $k > 0$ such that
			\[f^k(U)\underset{\Phi}{\cap} V \neq\emptyset. \]
				\end{definition}
			\noindent This definition provides a framework to determine whether two open sets share identical descriptive features under iteration, despite not intersect in a classical topological sense (see Figure \ref{fig:des}).
			\begin{figure}[H]
				\centering
				\includegraphics[width=0.95\linewidth]{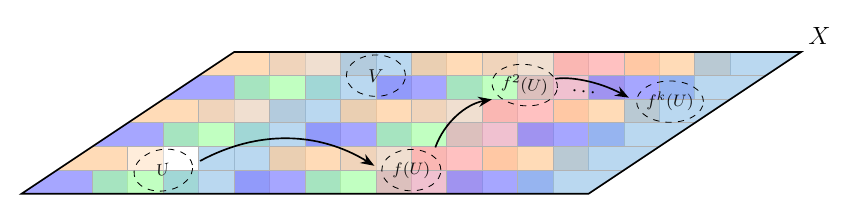}
				\caption{Descriptive transitivity.}
				\label{fig:des}
			\end{figure}

		\begin{lemma}	\label{lem22}
			
			Topological transitivity implies descriptive transitivity.
			\begin{proof}
				Let $X$ be a topological space, $(X, f, \Phi)$ descriptive  dynamical system and $U$ and $V$ be two non-empty open subsets of $X.$ Since $f$ is topologically transitive,  $f^k(U) \cap V\neq \emptyset$ for a postive integer $k>0$ . In this case, there exists a object $ b \in f^k(U) \cap V \subseteq f^k(U)\cup V $, which means that $ b $ is in both $f^k(U)$ and $V $. If $b \in f^k(U)$, then there exists $a \in U$ such that
				$f^k(a) = b$. Therefore, by equality $\Phi(f^k(a)) = \Phi(b)$ holds. Hence, $\Phi(f^k(U)) \cap \Phi(V) \neq \emptyset$, which implies that $f^k(U) \underset{\Phi}{\cap} V \neq \emptyset$.
			\end{proof}
		\end{lemma}

		% TODO: \usepackage{graphicx} required
		\begin{remark} \label{r1}
			The converses of Lemma~\ref{lem21} and Lemma~\ref{lem22} do not hold in general. To see this following Example \ref{exTQ} and Example \ref{exTI} are provided. Consider  the unit circle $S^1 \subset \mathbb{C}$ in the complex plane  and translations of circle map $T_\lambda : S^1 \rightarrow S^1$ , $T(\theta)=\theta + 2\pi\lambda$  (mod $2\pi)$ for $\lambda\in \mathbb{R}.$ This map has different properties when $\lambda$ is rational or irrational. 
			
				\begin{figure}[h]
				\centering
				\includegraphics[width=1\linewidth]{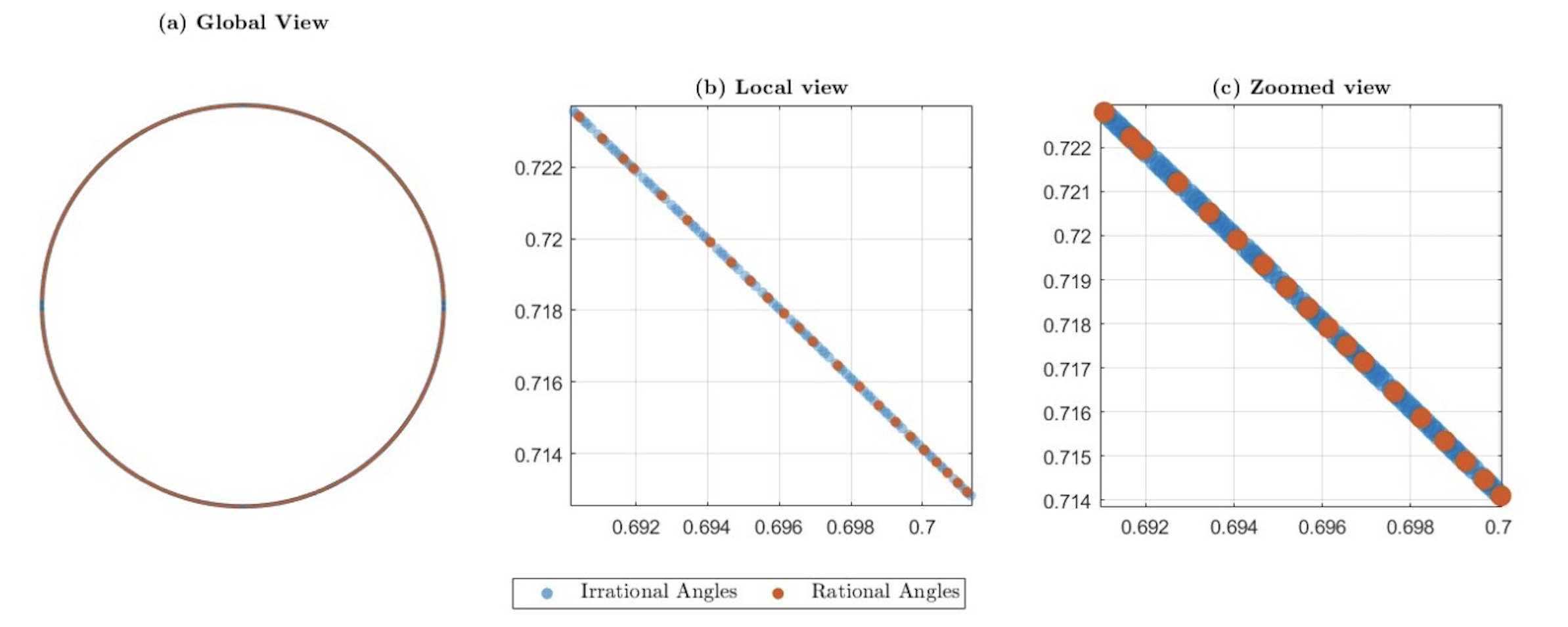}
				\caption{ Distributions of objects in $S^1$. }
				\label{fig:s2}
			\end{figure} 
			
	\noindent Let the unit circle $S^1$ be partitioned into four distinct sectors of equal length, defined as:
	\begin{equation*}
		W_1 = [0,\textstyle\frac{\pi}{2}), \quad W_2 = [\textstyle\frac{\pi}{2},\pi), \quad W_3= [\pi, \textstyle\frac{3\pi}{2}), \quad W_4= [\textstyle\frac{3\pi}{2}, 2\pi)
	\end{equation*}
	where  $S^1 = \bigcup_{i=1}^4 W_i$ and $W_i \cap W_j = \emptyset$ for $i \neq j$. Let define the probe function $\Phi: S^1 \rightarrow \mathbb{R}^n$,
	\begin{equation*}
		\Phi(\theta) = 
		\begin{cases} 
			v_1, & \text{if } \theta \in W_1 \\
			v_2, & \text{if } \theta \in W_2 \\
			v_3, & \text{if } \theta \in W_3 \\
			v_4, & \text{if } \theta \in W_4 
		\end{cases}
	\end{equation*}
		
			\end{remark}
		\noindent	in this respect,   the following examples will be examined.

\begin{example}\label{exTQ}

		Let $\lambda=\frac{1}{4}$, it is clear that  $\forall \theta \in S^1 $ is 4-periodic objects(points) of $T_{\frac{1}{4}}$ since 
	
		\begin{align*}
			T_{\frac{1}{4}}^4 (\theta) &= \theta + 4.(\frac{\pi}{2}) \; (\text {mod  } 2\pi) \\
			&=\theta + 2\pi \; (\text {mod  } 2\pi) \\
			& = \theta.
		\end{align*}
		
\noindent However, $T_{\frac{1}{4}}$ is not topologically transitive. To illustrate this  choose the open subsets $U=(\frac{-\pi}{20}, \frac{\pi}{20})$  and $V=(\frac{4\pi}{20},\frac{6\pi}{20})$ in $S^1 $. The first three iterations of $U$ are as follows

\[
U=(\frac{-\pi}{20}, \frac{\pi}{20}) \curvearrowright T(U)=(\frac{9\pi}{20},\frac{11\pi}{20}) \curvearrowright T^2(U)=(\frac{19\pi}{20}, \frac{21\pi}{20}) \curvearrowright T^3(U)= (\frac{29\pi}{20},\frac{31\pi}{20})
 \] and one can verify  that the iterates of $U$ never overlaps with $V$.
 On the other hand $T_{\frac{1}{4}}$ is descriptively transitive: To show this 
let $U$ and $V$ be non-empty open subsets of $S^1$ and define the following index sets:
\[
I_U= : \{ i \in \{1,2,3,4\} \; | \; U \cap W_i \neq \emptyset \}.
\]
 and 
 \[
I_V= :\{ i \in \{1,2,3,4\} \; | \; V \cap W_i \neq \emptyset \}
\]
Also, let $k_1=max \;I_U$, $k_2 = min\; I_V$ and set $
k= k_2 - k_1 \; (mod \;4$). Thus, for $\theta \in U\cap W_{k_1}$, we have  $T^k_{\frac{1}{4}}(\theta)\in W_{k_2}$. This implies    that  $T^k_{\frac{1}{4}}(U)$ and $V$ intersect within the same sector, which satisfies  
\[ \Phi(T^k_{\frac{1}{4}}(U)) \cap \Phi(V) \neq \emptyset. \] 
Consequently, $T^k_{\frac{1}{4}}(U) \underset{\Phi}{\cap }V \neq \emptyset$, that is,  $T_{\frac{1}{4}}$ is descriptively transitive.
  
		\end{example}
		\begin{example}\label{exTI} It is known that $T_{\lambda}$ does possess a dense set of periodic points for any irrational $\lambda$.
Let $\lambda$ be an irrational number such that \linebreak $\frac{1}{4}< \lambda < \frac{3}{4}$.  While $\lambda$ is irrational, the orbit  $\{T_\lambda^n(\theta) \: | \: \theta\in S^1, \; n\in \mathbb{Z}^+\}$ is dense in $S^1$, a result known  as Jacobi's Theorem. Let $U$ be non empty open subset of $S^1$. Since $U$ is non empty, there exists $i_0 \in \{1,2,3,4\}$ such that  $U\cap W_{i_0} \neq \emptyset$. Let  $\mathring{ W_{i_0}}$   denote the interior set of  $W_{i_0}$. Since  both $U$ and $ \mathring{ W_{i_0}}$ are open, for any $\theta \in U\cap \mathring{ W_{i_0}}$ there exists $\delta >0$ such that 
\[ B(\theta,\delta)\subset U\cap \mathring{ W_{i_0}}.
\]
where $B(\theta,\delta)= \{e^{it} \; | \; \theta-\delta < t < \theta +\delta \}$ is an open arc. Furthermore, because $B(\theta,\delta)$ is an open set and $\{T_\lambda^n(\theta) \: | \: \theta\in S^1, n\in \mathbb{Z}^+ \}$ is dense, there exists
$k \in \mathbb{Z}^+$ such that 
\[ 
T_\lambda^k(\theta) \in B(\theta,\delta) .
\]
Moreover, $\Phi(T_\lambda^k(\theta))= \Phi(\theta)=v_{i_0}\in \mathbb{R}^n $,  since both $T_\lambda^k(\theta)$ and $\theta$ belong to $	\mathring{ W_{i_0}}$. Consequently, $\theta$ is a descriptive $k$-periodic object of $T_\lambda$ in $U$.
  		\end{example}
  \noindent	Similarly, for suitable values of $\lambda$, the approach used in Example \ref{exTQ} and Example \ref{exTI} can be applied to $n$-equal distinct sectors $W_1,W_2, \dots W_n$.

			% TODO: \usepackage{graphicx} required

		\begin{definition}
			
			The function $f$ is said to be descriptively sensitive if there exists a number $\delta > 0$ such that for every object $a \in X$ and for every neighborhood $N$ of $a$, there exists an object $b$ in the neighborhood $N$ and a positive integer $k$ such that 
			 
			\[
			||\Phi(f^k(a)) - \Phi(f^k(b))|| > \delta.
			\]
		\end{definition}
\noindent While two trajectories may diverge significantly in terms of spatial distance, nevertheless they can  preserve their descriptive properties. The following remark is provided to illustrate this phenomenon.
		\begin{remark}	\label{doubling} 
			\noindent In metric spaces, sensitive dependence on initial conditions for a map $f$ does not necessarily imply descriptive sensitivity. 
				\end{remark}
				
	\noindent The following basic example can be given to illustrate this.			
				
\begin{example} \label{doublingsensitve}
Consider the doubling map $D$, a well-known example of a chaotic map on the unit circle, which is a subspace of the complex plane. This map is also known as the restriction of the map $F: \mathbb{C}\rightarrow \mathbb{C}, F(z)=z^2$,

			\[
			\begin{aligned}
				D: S^1 &\longrightarrow S^1 \\
				\theta &\longmapsto D(\theta) = 2\theta \pmod{2\pi}
			\end{aligned}
			\]
			together with constant probe function $\Phi: S^1\rightarrow \mathbb{R}^n,\Phi(\theta)=c\in \mathbb{R}^n$.
			\newline 
			For a fixed $\delta > 0$, let $\alpha \in S^1$ and $\beta \in N$ where $N$ is a neighborhood of $\alpha$ and $\alpha \neq \beta$. By the definition of $\Phi$
			\begin{eqnarray}
				||\Phi(D^k(\alpha))- \Phi(D^k(\beta)) ||  &=&  ||\Phi(2^k\alpha)- \Phi(2^k\beta)|| \nonumber\\ %% If equation numbering is not needed for a row use \nonumber.
				&=&||c-c||  \nonumber \\
				&=& 0 \ngtr \delta  \nonumber .
			\end{eqnarray}
		\end{example}
	
		\begin{corollary}\label{c1}
			Unlike the Banks' Theorem in the classical sense, the descriptive transitivity of a map and the density of its descriptive periodic objects set do not necessarily imply that  map is descriptively sensitive.
		\end{corollary}
			\subsection{A computational implementation}
% TODO: \usepackage{graphicx} required

 One of the encryption methods for images is Arnold’s Cat transformation, which rearranges the pixels of a given image \cite{alexan2025secure}. This transformation not only rearranges the pixels in a complex manner but also restores the original image after a specific number of iterations, known as the Arnold period, which depends on the image’s dimensions $N$ \cite{Dyson1992}. This transformation can be expressed as follows:

$$
\begin{pmatrix} 
	a_{n+1} \\ 
	b_{n+1} 
\end{pmatrix} = 
\begin{pmatrix} 
	1 & 1 \\ 
	1 & 2 
\end{pmatrix} 
\begin{pmatrix} 
	a_n \\ 
	b_n 
\end{pmatrix} \pmod{N}, 
$$
In this equation, $a$ and $b$ represent the coordinates of a pixel in the image. Arnold's Cat transformation employs matrix multiplication to stretch the $a$ and $b$ coordinates, followed by the modulo operator to fold and reconstruct the image \cite{Chen2004}. 
\noindent Consider an image of peppers with a resolution of $256\times 256$ pixels, taken by MATLAB library. In classical computational dynamics applications, the trajectories of pixels take center stage, while the descriptive perspective is often neglected. This results in a shallow approach to image analysis.  \noindent To show this, let $\Phi$ assign RGB color vectors to pixels and consider two pixels, purple colored $p_1 = (32, 32)$ and $p_2 = (32, 33)$, which rapidly diverge spatially under Arnold’s Cat Map transformation. 

 \begin{figure}[H]
	\centering
	\includegraphics[width=0.92\linewidth]{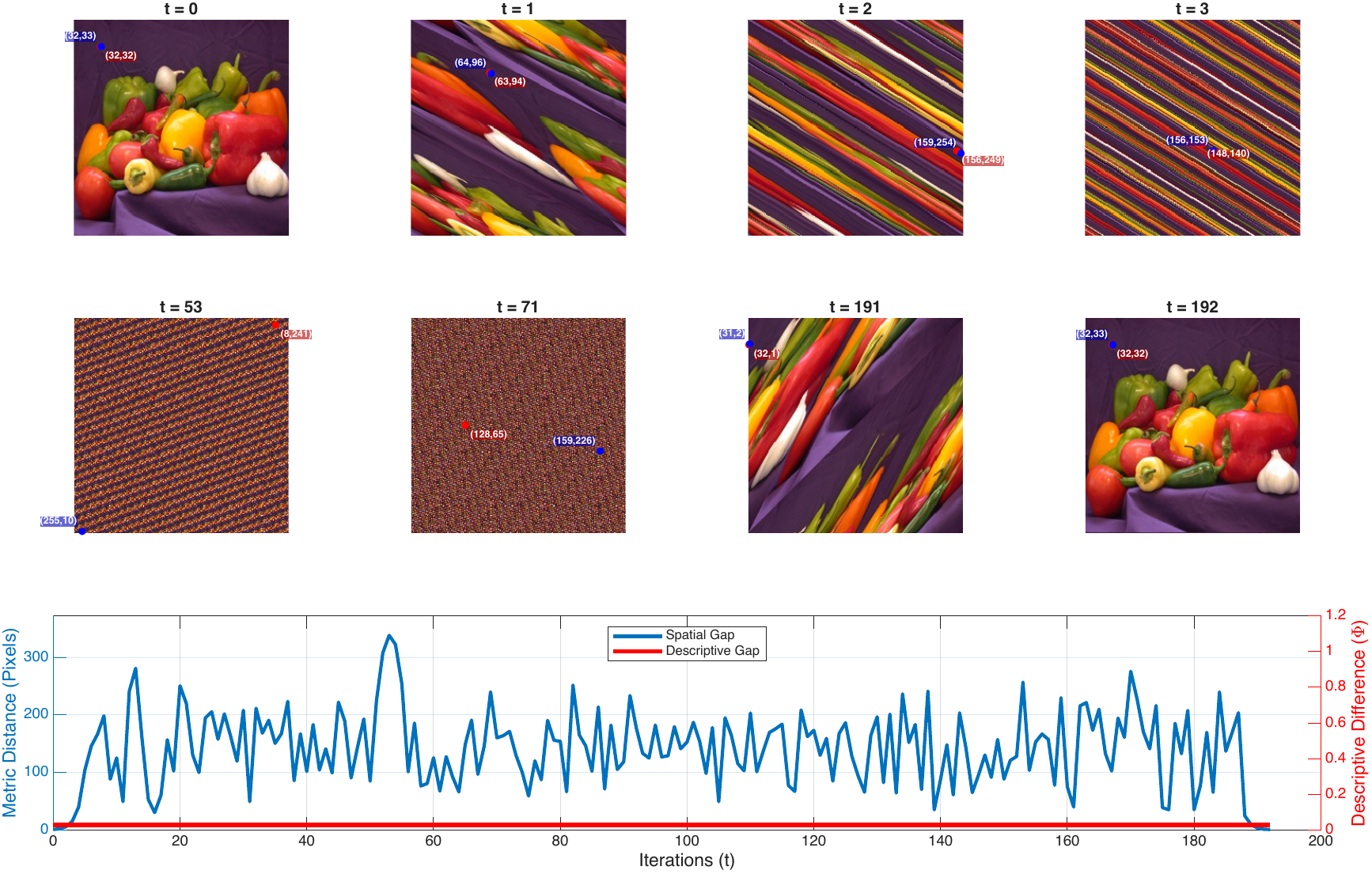}
	\caption{Classical approach to dynamical systems. }
	\label{fig:desstab}
\end{figure}
\noindent Despite the chaotic spatial evolution of the pixels, their descriptive attributes remain invariant due to the fixed background information. 
Thus, the system maintains a form of stability where proximity is preserved in the descriptive sense, even as metric distances increase between pixels. As illustrated in Figure~\ref{fig:desstab}, the descriptive gap remains constant despite a substantial increase in metric distance, most notably iteration at $t = 53$.  Consistent with the system's periodicity, both points restore their initial positions and values by iteration $t = 192$ just like others.\\

% TODO: \usepackage{graphicx} required 

\begin{figure}[H]
	\centering
	\includegraphics[width=0.92\linewidth]{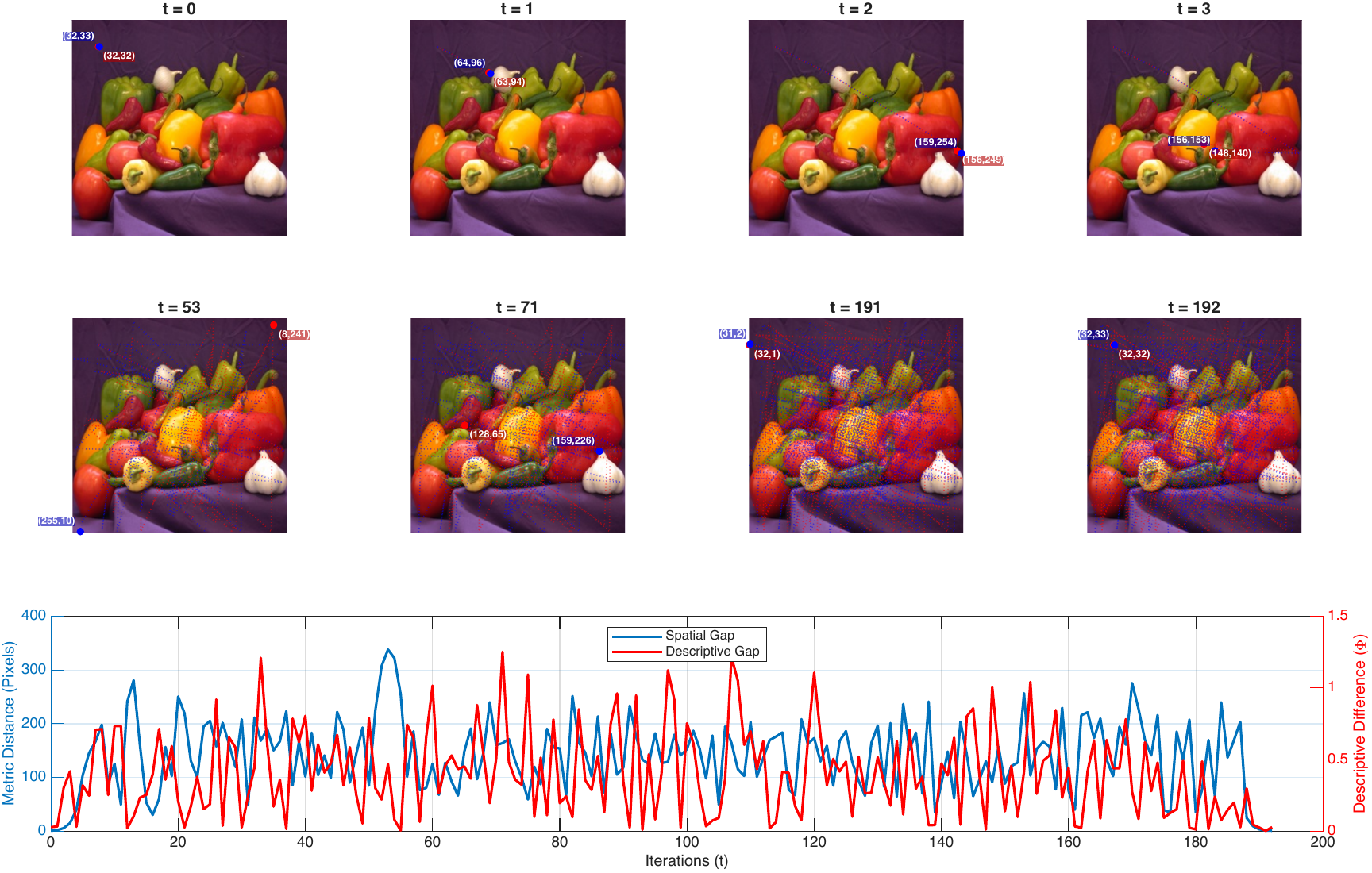}
	\caption{Topological chaos and descriptive instabilities.}
	\label{fig:desnstabb}
\end{figure}
\noindent Conversely, our theoretical framework examines both the spatial evolution and the descriptive values at each new coordinate as the pixels are iterated by the Arnold Cat Map. This perspective allows to evaluate whether or not descriptive proximity remains stable during the Arnold Cat Map iterations, even under significant spatial divergence.
\noindent By evaluating the descriptive value via using probe $\Phi$ at each iterated position of the pixels, a form of descriptive instability emerges. 
The gap between the descriptive values of these two moving pixels is especially maximized at iteration $t=71$. 

\noindent These implementations demonstrate that while the background remains the same, the descriptive relationship between the moving pixels becomes instable as they traverse the image see Figure \ref{fig:desnstabb}.
\noindent Consequently, this approach facilitates a more comprehensive image analysis by simultaneously capturing spatial  and descriptive instabilities.

		\section{Some results on topologically conjugate systems}
	Let $(X, f, \Phi$) and $(Y, g, \Psi)$ be descriptive topological dynamical systems. Assume that the functions $f$ and $g$ are topologically semi-conjugate via a continuous and surjective function $h: X \to Y$ (see Figure \ref{diagram}). In this situation
	the following lemmas are provided.
	
	\begin{lemma}
 If $f$ is descriptively transitive on $X$ and $h$  is  descriptively continuous, then $g$ is also descriptively transitive on $Y$.
\begin{proof}
Let $U_Y$ and $V_Y$ be non-empty open subsets of $Y$. Since the function $h$ is continuous and surjective,
$h^{-1}(U_Y) = U_X \subseteq X$ and $h^{-1}(V_Y) = V_X \subseteq X$ are non-empty and
open subsets. Since $f$ is descriptively transitive on $X$, there exists a  $k \in \mathbb{Z}^+$ such that $f^k(U_X) \underset{\Phi}{\cap} V_X\neq \emptyset$.  In this case  there exists    $a\in f^k(U_X) \cup V_X$  so that \[\Phi(a)\in \Phi( f^k(U_X) )\cap \Phi(V_X)\]
This also shows us that $f^k(U_X) \delta_{\Phi}V_X$. Then $h(a)\in h(f^k(U_X)) \cup h(V_X)$, by equality $h(f^k(U_X)\cup V_X)=h(f^k(U_X))  \cup h(V_X)$. Since $f$ and $g$ are topological semi-conjugate, 
\[h(a)\in g^k(h(U_X))\cup h(V_X)=g^k(U_Y)\cup V_Y \] and by our assumption  we get  \[g^k(U_Y)\delta_{\Psi} V_Y\]  
i.e., $\Psi(h(a)) \in \Psi(g^k(U_Y))\cap \Psi(V_Y)$, which means that $g^k(U_Y)\underset{\Psi}{\cap} V_Y\neq \emptyset$ as  desired.
 \end{proof}
		\end{lemma}

		\begin{lemma}
If the set of descriptive periodic objects of $f$ is dense in $X$ and $h$ is descriptively continuous, then the set of descriptive periodic objects of $ g$  is also dense in the space $Y $.
\begin{proof}
	Let $U_Y\subseteq Y$  be a non-empty open subset. Since $h$ is continuous and surjective $h^{-1}(U_Y)= U_X\subseteq X$ is non-empty and open.
	As the descriptive periodic objects of the function $f$ are dense in $X$, there exists an $a \in U_X$ such that $\Phi(f^k(a)) = \Phi(a)$ for some $k\in \mathbb{Z}^+$ . If $a \in U_X$ then $h(a) \in h(U_X)=U_Y$. Additionaly, because of  $f$ and $g$ are topological semi-conjugate and by our assumption,
	$\Psi(h(a))=\Psi (h(f^k(a))= \Psi(g^k(h(a)))$ equalitiy satisfies.
\end{proof}
		\end{lemma}
			\begin{figure}[H]
			\centering
			\includegraphics[width=0.6\linewidth]{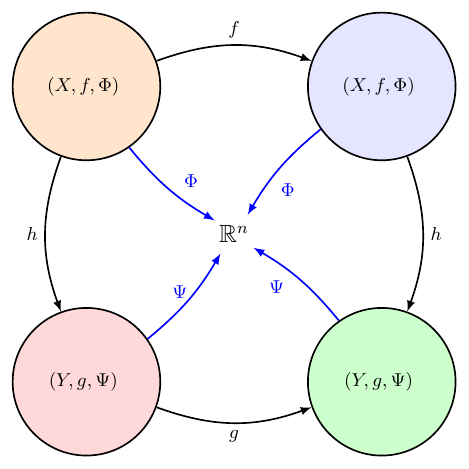}
			\caption{Conjugate systems.}
			\label{diagram}
		\end{figure}
		
	\noindent \textbf{Discussion} In contrast to the well-known Banks' theorem in classical topological dynamics, we show that descriptive transitivity and the density of descriptive periodic objects do not necessarily imply descriptive sensitivity. This distinction reveals that the descriptive proximity structure imposes a stricter or more specific constraint on the system’s behavior than the standard metric topology. Furthermore, descriptive continuity, a concept distinct from usual continuity, helps us to see that certain properties are preserved under topological conjugacy, as expected in a similar manner.


 % \cite{<label>} ==> [1]

%%

  

%\begin{thebibliography}{00}
	
	
	%\bibitem[Lamport(1994)]{lamport94}
	%Leslie Lamport,
	%\textit{\LaTeX: a document preparation system},
	%Addison Wesley, Massachusetts,
	%2nd edition,
	%1994.
	
%\end{thebibliography}

\bibliographystyle{elsarticle-num} 
\bibliography{referances} 

@article{banks1992devaney,
	title={On Devaney's definition of chaos},
	author={Banks, John and Brooks, Jeffrey and Cairns, Grant and Davis, Gary and Stacey, Peter},
	journal={The American mathematical monthly},
	volume={99},
	number={4},
	pages={332--334},
	year={1992},
	publisher={Taylor \& Francis}
}

@article{Di_Concilio2018-oy,
	title = {Descriptive proximities: {P}roperties and interplay between classical proximities and overlap},
	author = {Di Concilio, A. and Guadagni, C. and Peters, J. F. and Ramanna, S.},
	journal = {Mathematics in Computer Science},
	volume = {12},
	number = {1},
	pages = {91--106},
	month = {mar},
	year = {2018},
	publisher = {Springer}
}

@book{devaney2018introduction,
	title={An introduction to chaotic dynamical systems},
	author={Devaney, Robert},
	year={2018},
	publisher={CRC press}
}

@book{naimpally2013topology,
	title={Topology with applications: topological spaces via near and far},
	author={Naimpally, Somashekhar A and Peters, James F},
	year={2013},
	publisher={World Scientific}
}

@article{peters2007near,
	title={Near sets. General theory about nearness of objects},
	author={Peters, James F and others},
	journal={Applied Mathematical Sciences},
	volume={1},
	number={53},
	pages={2609--2629},
	year={2007}
}

@article{HAIDER2021111237,
	title = {Temporal proximities: Self-similar temporally close shapes},
	journal = {Chaos, Solitons \& Fractals},
	volume = {151},
	pages = {111237},
	year = {2021},
	issn = {0960-0779},
	author = {Muhammad Shangol Haider and James F. Peters}
}

@book{morris1989topology,
	title={Topology without tears},
	author={Morris, Sidney A},
	year={1989},
	publisher={University of New England}
}

@article{LIANG2025116826,
	title = {A new high-dimensional digital chaotic system and its S-box application},
	journal = {Chaos, Solitons \& Fractals},
	volume = {199},
	pages = {116826},
	year = {2025},
	issn = {0960-0779},
	author = {Runfa Liang and Qianxue Wang and Ying Li}
}

@article{MEZZI2025109571,
	title = {Chaos of trajectories},
	journal = {Topology and its Applications},
	volume = {375},
	pages = {109571},
	year = {2025},
	issn = {0166-8641},
	author = {Seif Mezzi and Khadija {Ben Rejeb}}
}

@book{book,
	author = {Peters, James},
	year = {2016},
	month = {04},
	pages = {},
	title = {Computational Proximity. Excursions in the Topology of Digital Images.},
	isbn = {ISBN 978-3-319-30260-7, ISBN 978-3-319-30262-1 (eBook)}
}

@article{alexan2025secure,
	title={A secure and efficient image encryption scheme based on chaotic systems and nonlinear transformations},
	author={Alexan, Wassim and Shabasy, Noura H El and Ehab, Noha and Maher, Engy Aly},
	journal={Scientific Reports},
	volume={15},
	number={1},
	pages={31246},
	year={2025},
	publisher={Nature Publishing Group UK London}
}

@article{Chen2004,
	title={A symmetric image encryption scheme based on 3D chaotic cat maps},
	author={Chen, Guanrong and Mao, Yongmao and Chui, Charles K.},
	journal={Chaos, Solitons \& Fractals},
	volume={21},
	number={3},
	pages={749--761},
	year={2004},
	publisher={Elsevier}
}

@article{Dyson1992,
title={The period of Arnold's cat map},
author={Dyson, Freeman J. and Falk, Harold},
journal={The American Mathematical Monthly},
volume={99},
number={7},
pages={603--614},
year={1992},
publisher={Taylor \& Francis}
}

@book{Arnold1968,
	title = {Ergodic Problems of Classical Mechanics},
	author = {Arnold, Vladimir I. and Avez, Andre},
	publisher = {W. A. Benjamin},
	address = {New York},
	year = {1968},
	note = {English translation of the original French edition}
}

@article{naimpally2013foreword,
	title={Foreword [near set theory and applications]},
	author={Naimpally, S and Peters, J and Wolski, M},
	journal={Math. Comput. Sci},
	volume={7},
	number={1},
	pages={1--2},
	year={2013}
}

@article{peters2013near,
	title={Near sets: An introduction},
	author={Peters, James F},
	journal={Mathematics in Computer Science},
	volume={7},
	number={1},
	pages={3--9},
	year={2013},
	publisher={Springer}
}
\end{document}